\documentclass[a4paper,10pt]{article}
\usepackage{mathtext}
\usepackage[T1,T2A]{fontenc}
\usepackage[cp1251]{inputenc}
\usepackage[english]{babel}
\usepackage{amsmath}
\usepackage{amsfonts}
\usepackage{amssymb}
\usepackage{mathrsfs}
\usepackage{amsthm}
\usepackage{titling}
\usepackage{multicol}
\usepackage{graphicx}
\usepackage{wrapfig}
\usepackage{appendix}
\usepackage{enumerate}
\usepackage{pb-diagram,lamsarrow,pb-lams}

\usepackage{euscript}

\usepackage{scalerel,stackengine}
\stackMath
\newcommand\reallywidehat[1]{%
\savestack{\tmpbox}{\stretchto{%
  \scaleto{%
    \scalerel*[\widthof{\ensuremath{#1}}]{\kern-.6pt\bigwedge\kern-.6pt}%
    {\rule[-\textheight/2]{1ex}{\textheight}}
  }{\textheight}%
}{0.5ex}}%
\stackon[1pt]{#1}{\tmpbox}%
}
\DeclareMathOperator{\I}{\mathcal{I}}
\newcommand{\T}{\mathbb{T}}

\DeclareMathOperator{\dist}{dist}

\newcommand{\W}{\mathfrak{W}}
\newcommand{\D}{\mathfrak{D}}
\newcommand{\F}{\mathcal{F}}
\newcommand{\AF}{\mathcal{AF}}

\newcommand{\J}{\mathrm{J}}

\renewcommand{\leq}{\leqslant}
\renewcommand{\geq}{\geqslant}

\newcommand{\scalprod}[2]{\langle{#1},{#2}\rangle}

\newtheorem{Le}{Lemma}[section]
\newtheorem{Def}[Le]{Definition}

\newtheorem{Th}[Le]{Theorem}

\newtheorem{Rem}[Le]{Remark}

\numberwithin{equation}{section}

\begin{document}
\author{Dmitriy~Stolyarov\thanks{Supported by RFBR grant 18-31-00037 and by ``Native towns'', a social investment program of PJSC ``Gazprom Neft''.}}
\title{Martingale interpretation of weakly cancelling differential operators}
\maketitle
\begin{abstract}
We provide martingale analogs of weakly cancelling differential operators and prove a Sobolev-type embedding theorem for these operators in the martingale setting.
\end{abstract}

\section{Preliminaries}
In~\cite{VS}, Van Schaftingen gave a characterization of linear homogeneous vector-valued elliptic differential operators~$A$ of order~$k$ in~$d > 1$ variables such that
the inequality
\begin{equation*}
\|\nabla^{k-1} f\|_{L_{\frac{d}{d-1}}(\mathbb{R}^d)} \lesssim \|A f\|_{L_1(\mathbb{R}^d)}
\end{equation*}
holds true for any smooth compactly supported function~$f$\footnote{The notation ``$X \lesssim Y$'' (as in the inequality above) means there exists a constant~$C$ such that~$X \leq CY$ uniformly. The parameter with regard to which we apply the term ``uniformly'' is always clear from the context.}. He called such operators cancelling. Let~$k \geq d$ and let~$l \in [1\,..\,d-1]$. It was also proved in~\cite{VS} that the operator~$A$ is cancelling (assuming the ellipticity) if and only if
\begin{equation*}
\|\nabla^{k-l} f\|_{L_{\frac{d}{d-l}}} \lesssim \|A f\|_{L_1}.
\end{equation*}
However, for the case~$l=d$, the cancellation condition is only sufficient. In~\cite{Raita}, Raita found a necessary and sufficient condition on the operator~$A$ for the inequality
\begin{equation*}
\|\nabla^{k-d} f\|_{L_{\infty}} \lesssim \|A f\|_{L_1}
\end{equation*}
to be true for any~$f \in C_0^{\infty}(\mathbb{R}^d)$. He called such operators weakly cancelling operators. 

The paper~\cite{ASW} suggested a martingale interpretation of Van Schaftingen's theorem. It appears that the cancellation condition has a direct analog in a probabilistic model earlier introduced in~\cite{Janson}. The present note provides an analog of Raita's weak cancellation condition.

We refer the reader to~\cite{ASW} for more history and motivation as well as for a more detailed description of the notation. See Section~\ref{SComparison} for comparison of our results with~\cite{Raita} and~\cite{VS}. 

The author thanks Bogdan Raita for attracting his attention to the question.

\section{Notation and statement}
Let~$m \geq 2$ be a natural number, let~$\F = \{\F_n\}_n$ be an~$m$-uniform filtration on a probability space. By this we mean that each atom of the algebra~$\F_n$ is split into~$m$ atoms of~$\F_{n+1}$ having equal probability. The symbol~$\AF_n$ denotes the set of all atoms in~$\F_n$. For each~$\omega \in \AF_n$, we fix a map
\begin{equation*}
\J_\omega \colon [1\,..\,m] \to \{\omega' \in \AF_{n+1}\mid \omega' \subset \omega\}.
\end{equation*}
This fixes the tree structure on the set of all atoms. Each atom in $\AF_n$ corresponds to a sequence of~$n$ integers in the interval~$[1\,..\,m]$, which we call digits. We may go further and consider the set~$\T$ consisting of all infinite paths in the tree of atoms. Each path starts from the atom in~$\F_0$, then chooses one of its sons in~$\F_1$, then one of its sons in~$\F_2$, and so on. There is a natural one-to-one correspondence between points in~$\T$, i.e. paths, and infinite sequences of digits in~$[1\,..\,m]$. There is also a natural metric on~$\T$. The distance between the two paths~$\gamma_1$ and~$\gamma_2$ is defined by the standard formula
\begin{equation}\label{Metric}
\dist(\gamma_1,\gamma_2) = m^{-d}, \quad d = \max\{n\mid \gamma_1(j) = \gamma_2(j) \hbox{ for all } j<n\}.
\end{equation}

Define the linear space~$V$ by the rule
\begin{equation*}
V = \Big\{x\in\mathbb{R}^m\;\Big|\, \sum\limits_{j=1}^m x_j = 0\Big\}.
\end{equation*}
Let~$\ell$ be an integer. We will be considering~$\mathbb{R}^\ell$-valued martingales adapted to~$\F$. Let~$F = \{F_n\}_n$ be an~$\mathbb{R}^\ell$-valued martingale. Define its martingale difference sequence by the rule
\begin{equation*}
f_{n+1} = F_{n+1} - F_n,\quad n \geq 0.
\end{equation*}
Now fix an atom~$\omega \in \AF_n$. The map~$\J_\omega$ may be naturally extended to the map that identifies an element of~$V\otimes \mathbb{R}^\ell$ with the restriction~$f_{n+1}|_{\omega}$ of a martingale difference to~$\omega$. In other words, the map~$\J_\omega$ identifies~$V\otimes \mathbb{R}^\ell$ with the space of~$\mathbb{R}^\ell$-valued~$\F_{n+1}$-measurable functions on~$\omega$ having mean value zero. The said extension will be also denoted by~$\J_\omega$. 

\begin{Def}
Let~$W \subset V\otimes \mathbb{R}^\ell$ be a linear subspace. Define the martingale Sobolev space by the rule
\begin{equation*}
\W = \Big\{F \hbox{ is an $L_1$-martingale}\;\Big|\,\forall n\quad \forall \omega \in \AF_n \quad f_{n+1}|_\omega \in \J_{\omega}[W]\Big\}.
\end{equation*}
The norm in~$\W$ is inherited from~$L_1$.
\end{Def}
We also introduce the martingale analog of the Riesz potential:
\begin{equation*}
\I_\alpha[F] = \Big\{\sum\limits_{k=0}^{n}m^{-\alpha k}f_{k}\Big\}_{n}, \quad \alpha > 0.
\end{equation*}
\begin{Th}[Theorem $1.9$ in \cite{ASW}]\label{OldMainTheorem}
If~$W$ does not contain non-zero rank-one tensors~$v\otimes a$ with~$v$ having~$m-1$ equal coordinates\textup, then
\begin{equation}\label{Embedding}
\|\I_{\frac{p-1}{p}}[F]\|_{L_p} \lesssim \|F\|_{\W}, \quad p \in (1,\infty].
\end{equation}
\end{Th}
\begin{Rem}
In fact\textup, a stronger inequality
\begin{equation}\label{StrongerEmbedding}
\sum\limits_{n \geq 0} m^{-\frac{p-1}{p}n}\|f_{n}\|_{L_p} \lesssim \|F\|_{\W}
\end{equation}
is true if~$W$ does not contain rank-one tensors~$v\otimes a$ with~$v$ having~$m-1$ equal coordinates \textup(see Theorem~$1.10$ in \textup{\cite{ASW}}\textup). Moreover\textup, the absence of the said vectors is also necessary for~\eqref{Embedding}.
\end{Rem}
It appears that if we put yet another martingale transform, the game becomes more interesting, at least for the endpoint case~$p=\infty$. Let~$\varphi\colon W\to V$ be a linear operator. When does the inequality
\begin{equation}\label{Raita's}
\Big\|\sum\limits_{n} m^{-n}\sum\limits_{\omega \in \AF_n}\J_\omega\Big[\varphi(\J^{-1}_\omega [f_{n+1}|_{\omega}])\Big]\Big\|_{L_{\infty}}\lesssim \|F\|_{\W}
\end{equation} 
hold true\footnote{There is a small inaccuracy in the notation here. Namely, the image of~$\J_\omega$ is formally defined as a function on~$\omega$. In the inequality above, we have extended it by zero to the remaining part of the probability space.}? By~\eqref{StrongerEmbedding} and the triangle inequality, it is true provided~$W$ does not contain rank-one tensors~$v\otimes a$ with~$v$ having~$m-1$ equal coordinates. Surprisingly,~\eqref{Raita's} may hold true in other cases. Seemingly, this effect is present for the case~$p=\infty$ only. 

Let~$D_1,D_2,\ldots, D_m$ be the ``nasty'' vectors in~$V$ that break our inequalities:
\begin{equation*}
D_j = (\underbrace{-1,-1,\ldots, -1}_{j-1}, m-1,-1,\ldots, -1).
\end{equation*}
\begin{Th}\label{MainTh}
The inequality~\eqref{Raita's} holds true if and only if 
\begin{equation}\label{WeakCancellation}
\big(\varphi(D_j\otimes a)\big)_j = 0
\end{equation}
whenever~$D_j\otimes a \in W$.
\end{Th}
Formula~\eqref{WeakCancellation} means that the~$j$-th coordinate of the vector~$\varphi[D_j\otimes a] \in V$ vanishes.

\section{Proof of Theorem~\ref{MainTh}}
\subsection{Necessity}
Assume there exists~$j\in [1\,..\, m]$ and a vector~$a\in \mathbb{R}^\ell \setminus \{0\}$ such that~$D_j\otimes a \in W$ and
\begin{equation*}
\big(\varphi(D_j\otimes a)\big)_j =\theta \ne 0.
\end{equation*}
Consider the martingale~$F$ defined as follows:
\begin{equation*}
F_n = a\cdot m^n\chi_{\omega_n}, \quad \hbox{where } \omega_n \in \AF_n, \quad n \geq 0,
\end{equation*}
corresponds to the sequence~$\{\underbrace{j,j,j\ldots,j}_n\}$. Then,
\begin{equation*}
f_{n+1} = \J_\omega [D_j\otimes a]\cdot m^n\chi_{\omega_n}.
\end{equation*}
Let us stop our martingale at the step~$N$ and plug the stopped martingale into~\eqref{Raita's}. Then, the sum on the left hand-side of~\eqref{Raita's} is equal to~$N\theta$ on the atom~$\omega_N$. So, the left hand-side tends to infinity as~$N\to \infty$, whereas the right hand-side is identically equal to one. So, if~$\theta \ne 0$, the inequality~\eqref{Raita's} cannot be true.

\subsection{Sufficiency}
\begin{Le}\label{LinearAlgebra}
Let~$G$ be a real finite dimensional linear space\textup, let~$E$ and~$F$ be its subspaces. Let~$\psi$ be a linear functional on~$E$\textup, which vanishes on~$E\cap F$. There exists a linear functional~$\Psi$ on~$G$ such that~$\Psi$ is an extension of~$\psi$ and it vanishes on~$F$.
\end{Le}
\begin{proof}
Consider the diagram
\begin{equation*}
\begin{diagram}
\node{G}
\arrow{e,t}{ }
\arrow{sseee,t}{\Psi}
\node{G/F}
\arrow[2]{se,t}{(2)}
\\
\node{E}
\arrow{n,t,J}{ }
\arrow{e,t}{ }
\arrow{seee,b}{\psi}
\node{E/(E\cap F)}
\arrow{n,t,J}{ }
\arrow{see,l}{(1)}\\
\node{ }
\node{ }
\node{ }
\node{\mathbb{R}}
\end{diagram}.
\end{equation*}
The arrow~$(1)$ exists because~$\psi|_{E\cap F} = 0$. The arrow~$(2)$ is constructed from~$(1)$ with the help of the Hahn--Banach theorem. The map~$\Psi$ is then restored by commutativity of the diagram. 
\end{proof}

We want to extend~$\varphi$ to the whole space~$V\otimes \mathbb{R}^\ell$ preserving the condition~\eqref{WeakCancellation}. For that, we consider coordinate functionals~$\varphi_j\colon W\to \mathbb{R}$, who are simply~$j$-th coordinates of~$\varphi$, and try to extend them. Consider the spaces~$\D_j$ defined as
\begin{equation*}
\D_j = \Big\{D_j\otimes a\;\Big|\,a \in\mathbb{R}^\ell\Big\}.
\end{equation*}
Formula~\eqref{WeakCancellation} means that~$\varphi_j|_{W\cap\D_j} = 0$ exactly. We apply Lemma~\ref{LinearAlgebra} with~$G := V\otimes \mathbb{R}^\ell$,~$E := W$,~$F:= \D_j$, and~$\varphi_j$ in the role of~$\psi$, and obtain a functional~$\Phi_j:= \Psi$ on~$V\otimes \mathbb{R}^\ell$, which vanishes on~$\D_j$ and extends~$\varphi_j$. Compose a linear operator~$\Phi\colon V\otimes \mathbb{R}^\ell\to \mathbb{R}^m$ from the functionals~$\Phi_j$:
\begin{equation*}
\Phi = (\Phi_1,\Phi_2,\ldots,\Phi_m).
\end{equation*}
This operator extends~$\varphi$ and satisfies the condition
\begin{equation}\label{WeakCancellationExtended}
\forall j\in [1\,..\,m] \quad \forall a\in\mathbb{R}^\ell\quad (\Phi[D_j\otimes a])_j = 0.
\end{equation}

It suffices to prove an \emph{a priori} stronger version of~\eqref{Raita's}:
\begin{equation}\label{Raita'sExtended}
\Big\|\sum\limits_{n} m^{-n}\sum\limits_{\omega \in \AF_n}\J_\omega\Big[\Phi(\J^{-1}_\omega [f_{n+1}|_{\omega}])\Big]\Big\|_{L_{\infty}}\lesssim \|F\|_{L_1}
\end{equation}
for any~$L_1$-martingale~$F$\footnote{Note that we have formally defined the map~$\J_\omega$ as a map on~$V$, and now we apply it to an element of~$\mathbb{R}^m$; this does not cause any problem though.}. We use the fact that any~$L_1$-martingale adapted to~$\F$ has the limit~$\mathbb{R}^\ell$-valued measure~$\mu$ of bounded variation on~$\T$ (the measure is defined on the Borel~$\sigma$-algebra on~$\T$ defined by the metric~\eqref{Metric}) related to~$F$ by the formula
\begin{equation}\label{MartMeasures}
F_n = \sum\limits_{\omega \in \AF_n} \mu(\omega)\cdot m^n\chi_{\omega}.
\end{equation}
So, the inequality~\eqref{Raita'sExtended} is an estimate of a linear operator on the space of measures. It suffices to verify it for the case where~$\mu$ is a delta measure. 

Let~$j = \{j_n\}_n$ be a sequence of digits, i.e. a point in~$\T$, let~$a\in \mathbb{R}^\ell$. Consider the martingale~$F$ that represents~$a\cdot\delta_{j}$ via formula~\eqref{MartMeasures}. In this case,
\begin{equation*}
f_{n+1} = \J_{\omega_n}\Big[D_{j_{n+1}}\otimes a\Big] \cdot m^n, \quad \hbox{ where } \omega_n = \{j_1,j_2,\ldots,j_n\}.
\end{equation*}

The condition~\eqref{WeakCancellationExtended} makes the summands in the inner sum in~\eqref{Raita'sExtended} have disjoint supports. Indeed,
\begin{equation*}
\J_{\omega_n}\Big[\Phi(\J^{-1}_{\omega_n} [f_{n+1}|_{\omega_n}])\Big] = \J_{\omega_n}\Big[\Phi(D_{j_{n+1}}\otimes a)\Big].
\end{equation*}
By~\eqref{WeakCancellationExtended}, this function is zero on the atom~$\{j_1,j_2,\ldots,j_n,j_{n+1}\}$, where all the functions~$f_{k}$ with~$k > n+1$ are supported.

Therefore,~\eqref{Raita'sExtended} follows from the trivial estimate~$\|f_{n+1}\|_{L_\infty} \lesssim m^n$.

\section{Comparison with the real-variable case}\label{SComparison}
Assume now that~$[1\,..\,m]$ is equipped with the structure of an abelian group~$G$. Let~$\Gamma$ be the dual group of~$G$. We may think of~$V$ and~$W$ as of spaces of functions on~$G$ having zero means\footnote{Since we will be working with the Fourier transform, one might wish to switch to complex scalars here. This does not lead to any problems.}. Assume further that~$W$ is translation invariant with respect to the action of~$G$. In this case, there exist spaces~$W_\gamma \subset \mathbb{R}^\ell$,~$\gamma\in \Gamma$, such that
\begin{equation*}
W = \Big\{w \in V\otimes \mathbb{R}^\ell\;\Big|\, \forall \gamma \in \Gamma\setminus \{0\} \quad \hat{w}(\gamma) \in W_{\gamma}\Big\}.
\end{equation*}
As it was proved in~\cite{ASW}, the condition that~$W$ does not contain rank-one tensors~$v\otimes a$ with$v$ having $m-1$ equal coordinates may be reformulated as
\begin{equation*}
\bigcap_{\gamma \in \Gamma \setminus \{0\}}W_{\gamma} = \{0\}.
\end{equation*}
This perfectly matches Van Schaftingen's cancelling condition in~\cite{VS}.

Let also the operator~$\varphi$ be translation invariant. This means there exist functionals~$\varphi_{\gamma}$ on the spaces~$W_{\gamma}$,~$\gamma \ne 0$, such that
\begin{equation*}
\widehat{\varphi[w]}(\gamma) = \varphi_{\gamma}[\hat{w}(\gamma)],\quad \gamma \in \Gamma\setminus \{0\}, \quad w \in W.
\end{equation*}
Let us express~\eqref{WeakCancellation} in Fourier terms using the Plancherel theorem (by translation invariance, it suffices to consider the case~$j=0$ only):
\begin{equation*}
\varphi[D_0\otimes a](0) = \scalprod{\varphi[D_0\otimes a]}{\delta_0}=\sum\limits_{\Gamma} \reallywidehat{\varphi[D_0\otimes a]}(\gamma) = \sum\limits_{\Gamma\setminus \{0\}} \varphi_{\gamma}[a],\quad D_0\otimes a \in W.
\end{equation*}
So, the condition~\eqref{WeakCancellation} is equivalent to
\begin{equation*}
\sum\limits_{\Gamma\setminus \{0\}} \varphi_{\gamma}[a] = 0, \quad \forall a \in\!\!\! \bigcap_{\gamma \in \Gamma\setminus \{0\}} W_{\gamma},
\end{equation*}
which perfectly matches Raita's weak cancelling condition in~\cite{Raita}.

D. Stolyarov

\medskip

Department of Mathematics and Computer Science, St. Petersburg State University, Russia

St. Petersburg Department of Steklov Mathematical Institute, Russia

\medskip

d.m.stolyarov@spbu.ru

\end{document}